\documentclass[a4paper]{article}
\usepackage{a4wide}
\usepackage[utf8]{inputenc}
\usepackage[english]{babel}
\usepackage{amssymb}
\usepackage{amsfonts}
\usepackage{amsmath}
\usepackage{amsthm}
\usepackage{graphicx}
\usepackage{hyperref}

\newtheorem{theorem}{Theorem}[section]
\newtheorem{lemma}{Lemma}[section]
\theoremstyle{remark}

\theoremstyle{definition}

\newcommand{\spec}{\operatorname{spec}}

\begin{document}

\begin{center}
\LARGE{Branching Random Walks with One Particle Generation Center and Possible Absorption at Every Point}
\end{center}
\begin{center}
\textit{E.\,Filichkina$^{1,2}$, E.\,Yarovaya$^1$}
\end{center}

\begin{center}
$^1$ Department of Probability Theory, Lomonosov Moscow State University,\\
$^2$ National  Medical  Research  Center  for  Therapy  and  Preventive  Medicine, \\ 
Moscow, Russia\\
\end{center}
\medskip

\begin{abstract}
We consider a new model of a branching random walk on a multidimensional lattice with continuous time and one source of particle reproduction and death, as well as an infinite number of sources in which, in addition to the walk, only absorption of particles can occur. The asymptotic behavior of the integer moments of both the total number of particles and the number of particles at a lattice point is studied depending on the relationship between the model parameters. In the case of the existence of an isolated positive eigenvalue of the evolution operator of the average number of particles, a limit theorem is obtained on the exponential growth of both the total number of particles and the number of particles at a lattice point.

\medskip

\end{abstract}
\sloppy
\section{Introduction} \label{introduction}
We consider a continuous-time branching random walk (BRW) on the multidimensional lattice $\mathbb{Z}^d$, ${d \in \mathbb{N}}$, with one source of particle reproduction and death located at the origin and an infinite number of \textit{absorbing} sources located at all other points of the lattice in which, in addition to walk, the particle can only disappear.

The behavior of a BRW with a single source of particle generation (\emph{branching}) located at the origin and no absorption at other points under the assumption of a finite variance of jumps has been studied, for example, in~\cite{YarBRW}, and with infinite variance in \cite{BRWht, BRWht1}. The random walk underlying the processes under consideration is defined using the transition intensity matrix ${A = \left(a(x,y)\right)_{x,y \in \mathbb{Z}^d}}$ and satisfies conditions of regularity, symmetry, spatial homogeneity (which allows us to consider $a(x, y)$ as a function of one argument $a(y-x)$), homogeneity in time and irreducibility. In these models the operator that specifies the evolution of the average number of particles has the form
\[
\mathcal{H} = \mathcal{A} + \beta \Delta_0,
\]
where the operator $\mathcal{A}: l^p(\mathbb{Z}^d) \rightarrow l^p(\mathbb{Z}^d)$ generated by the matrix $A$ acts on the function $\varphi\in l^p(\mathbb{Z}^d)$ by the formula
\begin{equation}\label{A:def}
	(\mathcal{A} \varphi)(x)=\sum_{y \in\mathbb{Z}^d}a(x-y)\varphi(y),\qquad x\in\mathbb{Z}^d,
\end{equation}
and the operator $\Delta_0$ is defined by the equality $\Delta_0 = \delta_0 \delta_0^T$,
where $\delta_0 = \delta_0(\cdot)$ denotes a column-vector on the lattice taking the unit value at the point $0 \in \mathbb{Z}^d$ and vanishing at other points. The parameter $\beta$ in the definition of the operator $\mathcal{H}$ is given by the equality $\beta:=\sum_{n>1} (n-1)b_n-b_{0}$, where $b_n$ is the intensity of occurrence of $n>1$ descendants of the particle, including the particle itself, $b_{0}$ is the absorption intensity of the particle. Thus, the operator $\beta \Delta_0$ determines the process of particle branching at the origin.

In a BRW with an infinite number of absorbing sources the evolution operator of the average number of particles is modified as follows
\[
\mathcal{E} = \mathcal{A} + \beta^* \Delta_0 - b_0 I,
\]
where $I$ is the identity operator and the last term specifies the process of absorption of particles at every lattice point.
Note that the parameter $\beta^*:=\sum_{n>1} (n-1)b_n$ in the considered BRW differs from the parameter $\beta=\beta^*-b_{0}$ in that for $b_{0}>0$ the parameter $\beta$ can take values from the interval $(-\infty, +\infty)$, while the parameter $\beta^*$ is non-negative: $\beta^* \geq 0$.

Let the parameter $\beta_c$ be determined by the formula $\beta_c := 1/G_0(0,0)$, where $G_\lambda(x,y)$ is the Green's function of the random walk. Many properties of the transition probabilities of a random walk $p(t, x, y)$ are expressed in terms of the Green's function, while the Green's function can be defined as the Laplace transform of the transition probability $p(t,x,y)$ by the formula:
\begin{equation}\label{Gf:def}
	G_\lambda(x,y) := \int_{0}^{\infty} e^{-\lambda t} p(t,x,y)\,dt,\quad \lambda \geq 0.
\end{equation}

As shown, for example, in \cite{YarBRW}, when the relation $\beta^* > \beta_c$ holds, the operator $\mathcal{A} + \beta^* \Delta_0$ has an isolated positive eigenvalue $\lambda_0$, which is the solution of the equation $\beta^* G_\lambda(0,0) = 1$.
The asymptotic behavior of the integer moments of the total number of particles and the number of particles at every point of the lattice in the process under consideration depends on the dimension of the lattice $d$, the relation between the parameters $\beta^*$ and $\beta_c$, and for $\beta^* > \beta_c $ also on the relation between $\lambda_0$ and~$b_0$.

In the case of $\beta > \beta_c$ a BRW with one source of particle generation and no absorbing sources is called \textit{supercritical}.
The operator $\mathcal{H}$ in this case has an isolated positive eigenvalue and there is an exponential growth in the number of particles at every point and in the total number of particles~\cite{YarBRW}.
In the process under consideration, if the relation $\beta^* > \beta_c$ holds, the operator $\mathcal{E}$ has an isolated eigenvalue ${\lambda_{\mathcal{E}} = \lambda_0 - b_0}$, where $\lambda_0 > 0$ is an isolated eigenvalue of the operator $\mathcal{A} + \beta^*\Delta_0$.
Note that the eigenvalue $\lambda_{\mathcal{E}}$ of the operator $\mathcal{E}$ is not always positive, so the behavior of the process differs significantly depending on the relation between the parameters $\lambda_0$ and $b_0$.

The structure of the paper is as follows.
In Section~\ref{description_of_the_model} we give a formal description of a BRW with particle reproduction at the origin and absorption at every point of the lattice. Section~\ref{key_eqs} presents the key equations. Section~\ref{the_first_moments} gives a complete classification of the asymptotic behavior of the first moments of particle numbers.
In Section~\ref{SupercriticalBRW} the limit Theorem~\ref{limth} is obtained, which states that despite the infinite number of absorbing sources an exponential growth of both the total number of particles and the number of particles at every point can be observed in the considered BRW. This happens when $\lambda_{\mathcal{E}}>0$, which is equivalent to $\lambda_0>b_0$.
In Section~\ref{CriticalBRW} we study the asymptotic behavior of the particle number moments for $\beta^* > \beta_c$ and $\lambda_{\mathcal{E}} = 0$ (${\lambda_0 = b_0}$), it is found that the integer moments both the total number of particles and the number of particles at every point grow in a power-law manner as $t \rightarrow \infty$, with the first moments behaving as constants at infinity.
In Section~\ref{SubcriticalBRW} we consider the remaining cases, that is, the case when $\beta^* > \beta_c$ and $\lambda_{\mathcal{E}} < 0$ (${\lambda_0 < b_0}$), and also, when the operator $\mathcal{E}$ does not have an isolated eigenvalue, that is, when $\beta^* \leq \beta_c$.
Theorems~\ref{thSubcritical1}, \ref{th_subcritical} and \ref{th_subcritical_ht} are obtained, stating that the moments of particle numbers in these cases decrease exponentially as ${t \rightarrow \infty}$.
It turned out that the results of Sections~\ref{SupercriticalBRW} and \ref{CriticalBRW} as well as Theorem~\ref{thSubcritical1} of Section~\ref{SubcriticalBRW} do not depend on the conditions imposed on the variance of random walk jumps, while the behavior of the process for $\beta^* \leq \beta_c$ turns out to be different for finite and infinite variance of jumps (Theorems~\ref{th_subcritical} and \ref{th_subcritical_ht}).

We will call the considered BRW \textit{supercritical} if $\beta^* > \beta_c$ and $\lambda_{\mathcal{E}} > 0$, \textit{critical} if $\beta^* > \beta_c$ and $\lambda_{\mathcal{E}} = 0$ and \textit{subcritical} if $\beta^* > \beta_c$ and $\lambda_{\mathcal{E}} < 0$ or $\beta^* \leq \beta_c$.

Note that there is no exponential decrease of moments in a BRW with a single source of particle generation (and the absence of other absorbing sources)~\cite{YarBRW}.
The classification of the asymptotic behavior of the BRW with possible absorption of particles at every point $\mathbb{Z}^{d}$ turns out to be closer to the classification of the behavior of the Markov branching process $\mu(t)$ with continuous time, where the average number of particles $\mathsf{E} \mu(t) = e^{at}$.
A branching process is called \textit{supercritical} if $\mathsf{E} \mu(t) > 1$ ($a > 0$), \textit{critical} if $\mathsf{E} \mu(t) = 1$ ($a = 0$) and \textit{subcritical} if $\mathsf{E} \mu(t) < 1$ ($a < 0$), that is, the average number of particles in the supercritical branching process increases exponentially, in the critical it tends to a constant and in the subcritical it decreases exponentially~\cite{sev}.


\section{Description of the Model}
\label{description_of_the_model}

Let us proceed to a formal description of the BRW with one source of particle reproduction and death located at the origin of coordinates and an infinite number of absorbing sources located at the remaining points of the lattice $\mathbb{Z}^d$, ${d \in \mathbb{N}}$.

The random walk underlying the process is specified using the transition intensity matrix ${A = \left(a(x,y)\right)_{x,y \in \mathbb{Z}^d}}$ and satisfies the conditions regularity, symmetry, spatial homogeneity (which allows us to consider  $a(x, y)$ as a function of one argument $a(y-x)$), time homogeneity and irreducibility (a particle can be at any point of the lattice).

The transition probability of a random walk, that is, the probability that at time $t \geq 0$ the particle is at point $y$, provided that at time $t=0$ it was at point $x$, is denoted by $p(t,x,y)$. Asymptotically for $h \rightarrow 0$ the transition probabilities are expressed in terms of the transition intensities as follows
\begin{align*}
	p(h,x,y) & = a(x,y)h + o(h), \quad x \neq y, \\
	p(h,x,x) & = 1 + a(x,x)h + o(h).
\end{align*}

Note that the condition for the finite variance of jumps in terms of the transition intensity matrix is written as $\sum_{z \in \mathbb{Z}^d} |z|^2 a(z) < \infty$.
In situations where the finiteness of the variance of jumps turns out to be essential we will separately consider the case when the function $a(z)$ has the following behavior at infinity
\begin{equation}\label{HT:def}
	a(z) \sim \frac{H(z/|z|)}{|z|^{d+\alpha}}, \quad |z| \rightarrow \infty,
\end{equation}
where ${|\cdot|}$ is Euclidean norm on $\mathbb{R}^d$, ${H(z/ |z|) = H(-z/|z|)}$ is a positive continuous function on ${\mathbb{S}^{d-1} = \{z \in \mathbb{R}^d :|z| = 1\}}$, ${\alpha \in (0,2)}$ and the symbol $\sim$ here and below will denote the asymptotic equivalence of functions. Under this assumption the variance of jumps becomes infinite (see \cite{l1}). Random walks with infinite variance of jumps are commonly referred to in the literature as random walks with heavy tails.
We will consider the simplest case, when ${H(z/ |z|) \equiv C >0}$, and use the results obtained in \cite{BRWht, BRWht1}, where a BRW with one particle generation center and the absence of absorbing sources was considered under condition~\eqref{HT:def}.

To describe the behavior of a random walk it is convenient to use the Green's function $G_\lambda(x,y)$, which, as mentioned in the introduction, can be defined as the Laplace transform of the transition probability $p(t,x,y)$ by the formula \eqref{Gf:def}.

As in \cite{YarBRW} we will call the random walk \textit{recurrent} if ${G_0(0,0) = \infty}$ and \textit{nonrecurrent} or \textit{transient} if ${G_0(0,0) < \infty}$. In the case of finite variance of jumps the random walk is transient for $d \geq 3$ and is recurrent for $d = 1,2$, while in the case of infinite variance of jumps (when the condition~\eqref{HT:def} is satisfied) the transience of a random walk turns out to be possible in the dimension ${d = 1}$ for ${\alpha \in (0,1)}$ and in the dimension ${d = 2}$ for ${\alpha \in (0,2)}$.

The branching process at the particle generation center is specified using the infinitesimal generating function $f(u) = \sum_{n = 0}^{\infty} b_{n} u^n$, $0 \leq u \leq 1$, where $b_{n} \geq 0$ for $n \neq 1$, $b_{1} < 0$, $\sum_{n = 0}^{\infty}b_{n} = 0$. The coefficients $b_{n}$ determine the main linear part of the probability $p_*(h,n)$ of having $n$ particles at time $h$ provided that there was one particle at the initial time $t=0$:
\begin{align*}
	p_*(h,n) & = b_n h + o(h) \text{ for } n \neq 1, \\
	p_*(h,1) & = 1+b_1 h + o(h).
\end{align*}
The coefficients $b_{n}$ for $n \geq 1$ can be interpreted as the intensities of appearance of $n$ descendants of the particle, including the particle itself, while $b_{0}$ is interpreted as the intensity of death, or absorption, of the particle.
The generating function at other points of the lattice has a simpler form: $\overline{f}(u) = b_0 + \overline{b}_1u = b_0(1 - u)$. Further, we assume that the intensity of death is the same at all lattice points.

The evolution of particles in the system occurs as follows: a particle located at some time ${t>0}$ at the point ${x \in \mathbb{Z}^d}$ in a short time ${dt \rightarrow 0}$ can either jump to the point ${y \neq x,y \in \mathbb{Z}^d}$, with probability $a(x,y)dt + o(dt)$, or die with probability $b_0 dt + o(dt)$. If the point $x$ is the center of particle generation $(x = 0)$, then the particle can also produce $n>1$ descendants, including itself, with probability $b_n dt + o(dt)$. Otherwise, with probability $1 + a(x,x)dt + \delta_0(x)b_1 dt + (1-\delta_0(x))(-b_0 dt) + o(dt)$, the particle remains at the point $x$ during the entire time interval $[t, t+dt]$. We assume that each new particle evolves according to the same law, independently of other particles and of the entire prehistory.

The main objects of study in BRW are the number of particles at the time ${t \geq 0}$ at the point $y \in \mathbb{Z}^d$ (the local number of particles), denoted by $\mu(t, y)$, the total number of particles (particle population), denoted by ${\mu(t) = \sum_{y \in \mathbb{Z}^d} \mu(t, y)}$, and their integer moments, which are denoted as ${m_n(t,x,y) := \mathsf{E}_x \mu^n(t,y)}$ and ${m_n(t, x) := \mathsf{E}_x \mu^n(t)}$, ${n \in \mathbb{N}}$, where $\mathsf{E}_x$ is the mean on condition ${\mu(0, y) = \delta(x - y)}$, $\delta(\cdot)$ is the Kronecker delta on $\mathbb{Z}^d$. We will assume that at the initial moment of time $t=0$ the system consists of one particle located at the point ${x \in \mathbb{Z}^{d}}$, so the expectations of the local and total number of particles satisfy the initial conditions ${m_1(0,x,y) = \delta_y(x)}$ and ${m_1(0, x) \equiv 1}$ respectively.

\section{Key Equations}
\label{key_eqs}

Let us present the key equations that will be required to study the behavior of the considered BRW. The proofs of the theorems presented in this Section are based on the methods developed in \cite{YarBRW} and follow the same scheme, so the corresponding theorems will be presented below without proof.

We introduce the Laplace generating functions of the random variables $\mu(t,y)$ and $\mu(t)$ for $z \geq 0$:
\[
F(z;t,x,y) := \mathsf{E}_x e^{-z \mu(t,y)}, \quad F(z;t,x) := \mathsf{E}_x e^{-z \mu(t)}.
\]

Taking into account the evolution of particles in the system and using the Markov property of the process, the following statement can be proved for the generating functions.

\begin{theorem}
	The functions $F(z;t,x)$ and $F(z;t,x,y)$ are continuously differentiable with respect to $t$ uniformly with respect to ${x,y \in \mathbb{Z}^d}$ for all ${0 \leq z \leq \infty}$. They are the solutions to the following Cauchy problems:
	\begin{align*}
		\partial_t F(z;t,x)   & = (\mathcal{A}F(z;t,\cdot))(x) +
		\delta_0(x)f(F(z;t,x)) +                                        \\
		& \quad+  (1-\delta_0(x))b_0(1-F(z;t,x)), \\
		\partial_t F(z;t,x,y) & = (\mathcal{A}F(z;t,\cdot,y))(x) +
		\delta_0(x)f(F(z;t,x,y)) +                                      \\
		& \quad+ (1-\delta_0(x))b_0(1-F(z;t,x,y))
	\end{align*}
	with the initial conditions $F(z;0,x) = e^{-z}$ and $F(z;0,x,y) = e^{-z\delta_y(x)}$ respectively. Here $\mathcal{A}: l^p(\mathbb{Z}^d) \rightarrow l^p(\mathbb{Z}^d)$, $1 \leq p \leq \infty$, is a walk operator that acts on the function $\varphi\in l^p(\mathbb{Z}^d)$ by the formula~\eqref{A:def}.
\end{theorem}

Note that the proof of this theorem repeats the arguments from the proof of Lemma 1.2.1 in~\cite{YarBRW} and differs only in technical details.

The following theorem turns out to be true for the moments of particle numbers.

\begin{theorem} \label{difeq_m}
	The moments $m_n (t,\cdot,y) \in l^2(\mathbb{Z}^d)$ and $m_n (t,\cdot) \in l^{\infty}(\mathbb{Z}^d)$ satisfy the following differential equations in the corresponding Banach spaces for all natural $n \geq 1$:
	\begin{align}
		\label{eq1}
		\frac{d m_1}{dt} & = \mathcal{E}m_1 = \mathcal{A}m_1 + \beta^*
		\Delta_0 m_1- b_0 m_1,                                                                         \\
		\label{eq2}
		\frac{d m_n}{dt} & = \mathcal{E}m_n+ \delta_0(\cdot)g_n(m_1, \ldots, m_{n-1}), \quad n \geq 2,
	\end{align}
	with the initial conditions $m_n(0,\cdot,y) = \delta_y(\cdot)$ and $m_n(t, \cdot) \equiv 1$ respectively.
	Here
	$\beta^* := \sum_{n > 1} (n-1) b_n$, the operator $\mathcal{A}: l^p(\mathbb{Z}^d) \rightarrow l^p(\mathbb{Z}^d)$ is given by the formula~\eqref{A:def}, the operator $\Delta_0$ is defined by the equality ${\Delta_0 = \delta_0 \delta_0^T}$, where ${\delta_0 = \delta_0(\cdot)}$ denotes a column-vector on the lattice taking the unit value at the point $0 \in \mathbb{Z}^d$ and vanishing at other points and the function $g_n(m_1,\ldots,m_{n-1})$ is given by the formula
	\[
	g_n(m_1,\ldots,m_{n-1}) := \sum_{r = 2}^{n} \frac{\beta^{(r)}}{r!} \sum_{\substack{i_{1}, \ldots, i_{r}>0 \\ i_{1}+\cdots+i_{r}=n}} \frac{n!}{i_1! \cdots i_r!} m_{i_1} \cdots m_{i_r},
	\]
	where $\beta^{(r)} := f^{(r)}(1)$.
\end{theorem}

The proof of this theorem repeats the argument of the proof of Theorem 1.3.1 from~\cite{YarBRW}. It also uses equations for generating functions, the Fa\`a di Bruno's formula and the following property:
\begin{align*}
	m_n(t,x)   & = (-1)^n \lim_{z \rightarrow 0+} \partial_z^n F(z;t,x),   \\
	m_n(t,x,y) & = (-1)^n \lim_{z \rightarrow 0+} \partial_z^n F(z;t,x,y).
\end{align*}

Consider separately the case ${\beta^* = 0}$, this condition is equivalent to the fact that all $b_n$ for ${n>1}$ are equal to zero.
That is, in this case the particle does not produce new descendants and only the death and movement of the particle along the lattice is possible.
The operator describing the evolution of the average number of particles in this particular case has the form ${\mathcal{E} = \mathcal{A} - b_0I}$ and the equations for the moments for all ${n \in \mathbb{N}}$ take the form
\[
\partial_t m_n = \mathcal{A}m_n - b_0 m_n.
\]
Making the change of variables $m_n = q_n e^{-b_0t}$ in the last equation, we get that the functions $q_n$ satisfy the equation
\[
\partial_t q_n = \mathcal{A}q_n.
\]
The equation for the transition probabilities of a random walk $p(t,x,y)$ has the same form,
whence we get that
\[
m_n(t,x,y) = e^{-b_0t}p(t,x,y), \qquad m_n(t,x) =  e^{-b_0t},
\]
for all $d, n \in \mathbb{N}$.

Further, we will assume that the parameter $\beta^*$ is strictly positive (a particle in the generation source can produce at least one new particle).

Integral equations for the moments will play an important role in the further analysis, the derivation of which is carried out according to the same scheme as in \cite[Theorem 1.4.1]{YarBRW}.

\begin{theorem} \label{th_int}
	The moment $m_1(t,x,y)$ satisfies both integral equations
	\begin{align}
		\label{eq3} m_1(t,x,y) & = p(t,x,y) + \int_{0}^{t} (\beta^* p(t-s, x, 0) - b_0
		e^{\mathcal{A}(t-s)}) m_1(s,0,y)\,ds,                                                                                \\
		\label{eq4} m_1(t,x,y) & = p(t,x,y) + \int_{0}^{t} (\beta^* p(t-s, 0, y) - b_0 e^{\mathcal{A}(t-s)}) m_1(s,x,0)\,ds.
	\end{align}
	
	The moment $m_1(t,x)$ satisfies both integral equations
	\begin{equation}
		\label{eq5}
		\begin{aligned} m_1(t,x) & = 1 + \int_{0}^{t} (\beta^* p(t-s, x, 0) - b_0
			e^{\mathcal{A}(t-s)}) m_1(s,0)\,ds,                                                \\
			m_1(t,x) & = 1 + \int_{0}^{t} (\beta^* - b_0 e^{\mathcal{A}(t-s)}) m_1(s,x,0)\,ds.
		\end{aligned}
	\end{equation}
	
	For $k > 1$ the moments $m_k(t,x,y)$ and $m_k(t,x)$ satisfy the equations
	\begin{equation}
		\label{eq6}
		\begin{aligned} m_k(t,x,y) & = m_1(t,x,y)+                    \\&\quad + \int_{0}^{t} m_1(t-s, x,
			0)g_k(m_1(s,0,y), \dots, m_{k-1}(s,0,y))\,ds, \\
			m_k(t,x)   & = m_1(t,x)+                      \\&\quad + \int_{0}^{t} m_1(t-s, x, 0)g_k(m_1(s,0), \dots, m_{k-1}(s,0))\,ds.
		\end{aligned}
	\end{equation}
\end{theorem}

Note that the derivation of the differential and integral equations presented in this Section does not depend on the conditions imposed on the variance of random walk jumps, as noted, for example, in~\cite{limth,BRWht1}.

\section{Classification of the Asymptotic Behavior of the First Moments}
\label{the_first_moments}

Let us first study the asymptotic behavior of the first moments. To do this we pass from the functions $m_1(t,\cdot,y)$ and $m_1(t,\cdot)$ to the functions $q(t,\cdot,y)$ and $q(t,\cdot)$, making a change of variables ${m_1 = qe^{-b_0 t}}$. We obtain an equation for the functions $q(t,\cdot,y)$ and $q(t,\cdot)$ of the form
\[
\frac{d q}{dt} = \mathcal{A}q + \beta^* \Delta_0 q
\]
with the initial conditions $q(0,\cdot ,y) = \delta_y(\cdot)$ and $q(0, \cdot) \equiv 1$ respectively.

Note that the resulting equation has exactly the same form as the equation for the first moments in the BRW without absorbing sources, considered in~\cite{YarBRW} (or in~\cite{BRWht} for the case of heavy tails), which greatly simplifies the study.
The classification of the asymptotic behavior of the first moments of the local number of particles and the total number of particles for arbitrary $d-$dimensional lattices in the considered BRW can be obtained using the classification of the asymptotic behavior for the functions $q(t,x,y)$ and $q(t,x)$, obtained in~\cite{YarBRW},\cite{BRWht}, and the relation $m_1 = qe^{-b_0 t}$.

As in~\cite{YarBRW} we denote ${\beta_c := 1/G_0(0,0)}$, where $G_\lambda(x,y)$
is the Green's function of the random walk. When ${\beta^* > \beta_c}$ the operator
$\mathcal{A} + \beta^* \Delta_0$ has a single isolated positive eigenvalue $\lambda_0$, which is a solution of the equation $\beta^* G_\lambda(0,0) = 1$. However, the eigenvalue $\lambda_{\mathcal{E}}$ of the operator $\mathcal{E}$ that arises in this case is equal to $\lambda_0 - b_0$ and is not always positive, which complicates the problem. In contrast to the BRW considered in~\cite{YarBRW}, the asymptotic behavior of the process considered in this paper differs significantly depending on the relation between the parameters $\lambda_0$ and $b_0$, namely, for ${\lambda_0 > b_0}$, $\lambda_0 = b_0$ and $\lambda_0 < b_0$.

So, in the case of a finite variance of jumps we obtain the following classification of the asymptotic behavior of the first moments.

\begin{theorem} \label{th_fm}
	Let the variance of jumps of the random walk be finite, then for ${t \rightarrow \infty}$ the asymptotic behavior of the first moments can be represented as
	\[
	m_1(t,x,y) \sim C(x,y)u^*(t), \quad m_1(t,x) \sim C(x)v^*(t),
	\]
	where $C(x,y), C(x)$ are some positive functions, whose explicit form was obtained in~\cite{YarBRW}, and the functions $u^*(t)$ and $v^*(t)$ have the following form
	\begin{itemize}
		\item[$a)$] for $\beta^* > \beta_c$: $u^*(t) = e^{\lambda_{\mathcal{E}} t}$, $v^*(t) = e^{\lambda_{\mathcal{E}} t}$;
		
		\item[$b)$] for $\beta^* = \beta_c$:
		
		\begin{itemize}
			\item[] $d=3$: $u^*(t) = t^{-1/2}e^{-b_0 t}$, $v^*(t) = t^{1/2}e^{-b_0 t}$;
			\item[] $d=4$: $u^*(t) = (\ln t)^{-1}e^{-b_0 t}$, $v^*(t) = t(\ln t)^{-1}e^{-b_0 t}$;
			\item[] $d \geq 5$: $u^*(t) = e^{-b_0 t}$, $v^*(t) =t e^{-b_0 t}$;
		\end{itemize}
		
		\item[$c)$] for $\beta^* < \beta_c$, $d \geq 3$: $u^*(t) = t^{-d/2}e^{-b_0 t}$, $v^*(t) = e^{-b_0 t}$.
	\end{itemize}

\end{theorem}

Note that for a recurrent random walk ${\beta_c = 0}$, and since the parameter $\beta^*$ is assumed to be positive, then assuming a finite variance of jumps for ${d \leq 2}$ the relation ${\beta^* > \beta_c}$ always holds, due to which in the above classification, in contrast to~\cite{YarBRW}, there are no cases of ${d = 1,2}$ for ${\beta^* \leq \beta_c}$.

We also note that for ${\beta^* \leq \beta_c}$ for all $d$ an exponential decrease in the first moments of both the local number and the total number of particles is observed.

Let us separately consider the result obtained for ${\beta^* > \beta_c}$. In this case, since $\lambda_{\mathcal{E}} = \lambda_0 - b_0$, the asymptotic behavior of the first moments depends on the relation between $\lambda_0$ and $b_0$: three different cases are possible. For ${\lambda_0 > b_0}$ an exponential growth of the first moments is observed, for ${\lambda_0 = b_0}$ the first moments tend to a constant and for ${\lambda_0 < b_0}$ an exponential decrease is observed, these cases correspond to supercritical, critical and subcritical cases in the theory of branching processes~\cite{sev}.

The classification of the asymptotic behavior of the first moments in the case of heavy tails uses the classification of the behavior of the functions $q(t,x,y)$ and $q(t,x)$ obtained in~\cite{BRWht}.

\begin{theorem} \label{th_fm_ht}
	Under the condition~\eqref{HT:def} the asymptotic behavior of the first moments for ${\alpha \in (0,2)}$ and ${t \rightarrow \infty}$ can be represented as
	\[
	m_1(t,x,y) \sim C(x,y)u^*(t), \quad m_1(t,x) \sim C(x)v^*(t),
	\]
	where ${C(x,y), C(x) > 0}$ and the functions $u^*(t)$ and $v^*(t)$ have the following form
	
	\begin{itemize}
		
		\item[$a)$] for $\beta^* > \beta_c$: $u^*(t) = e^{\lambda_{\mathcal{E}} t}$, $v^*(t) = e^{\lambda_{\mathcal{E}} t}$;
		
		\item[$b)$] for $\beta^* = \beta_c$:
		\begin{itemize}
			\item[]  $u^*(t) = t^{d/\alpha - 2}e^{-b_0 t}$, $v^*(t) = t^{d/\alpha - 1}e^{-b_0 t}$, if $d/\alpha \in (1,2)$;
			\item[]  $u^*(t) = (\ln t)^{-1}e^{-b_0 t}$, $v^*(t) = t (\ln t)^{-1} e^{-b_0 t}$, if $d/\alpha = 2$;
			\item[]  $u^*(t) = e^{-b_0 t}$, $v^*(t) = t e^{-b_0 t}$, if $d/\alpha \in (2, +\infty)$;
		\end{itemize}
		
		\item[$c)$] for $\beta^* < \beta_c$: $u^*(t) = t^{-d/\alpha}e^{-b_0 t}$, $v^*(t) = e^{-b_0 t}$, $d/\alpha \in (1, +\infty)$.
		
	\end{itemize}
\end{theorem}

Note that for ${\beta^* > \beta_c}$ the obtained asymptotic relations do not depend on the conditions imposed on the variance of random walk jumps (see \cite{limth}). In addition, $\beta^*>0$, while $\beta_c = 0$ for $d/\alpha \in (1/2,1]$, so in the above classification for ${\beta^* \leq \beta_c}$ there are no cases where $d/\alpha \in (1/2,1]$, in contrast to the classification of the asymptotic behavior of the first moments in~\cite{BRWht}.

\section{Supercritical Case}\label{SupercriticalBRW}
\begin{theorem} \label{th8}
	Let ${\beta^* > \beta_c}$ and $\lambda_{\mathcal{E}} > 0$. Then for ${t \rightarrow \infty}$ and all ${n \in \mathbb{N}}$ the following statements hold:
	\[
	m_n(t,x,y) \sim C_n(x,y)e^{n\lambda_{\mathcal{E}}t}, \qquad m_n(t,x) \sim C_n(x)e^{n\lambda_{\mathcal{E}}t},
	\]
	where
	\[
	C_1(x,y) = \frac{G_{\lambda_0}(x,0) G_{\lambda_0}(0,y)}{\|G_{\lambda_0}(0,y)\|^2},  \qquad C_1(x) = \frac{G_{\lambda_0}(x,0)}{\lambda_0 \|G_{\lambda_0}(0,0)\|^2},
	\]
	and the functions ${C_n(x,y), C_n(x) > 0}$ for ${n \geq 2}$ are defined as follows:
	\begin{align*}
		C_n(x,y) & = g_n(C_1(0,y), \dots, C_{n-1}(0,y)) D_n(x), \\
		C_n(x)   & =  g_n(C_1(0), \dots, C_{n-1}(0)) D_n(x),
	\end{align*}
	where $D_n(x)$ are certain functions satisfying the estimate $|D_n(x)| \leq \frac{2}{n\lambda_{\mathcal{E}}}$ for $n \geq n_*$ and some ${n_* \in \mathbb{N}}$.
\end{theorem}

\begin{proof}
	In the case under consideration the operator $\mathcal{E}$ has an isolated positive eigenvalue $\lambda_{\mathcal{E}} = \lambda_0 - b_0$, where $\lambda_0$ is an isolated positive eigenvalue of the operator $\mathcal{H} = \mathcal{A} + \beta^* \Delta_0$.
	
	For ${n \in \mathbb{N}}$ we consider the functions $\nu_n := \nu_n (t,x,y) = m_n(t,x,y) e^{-n\lambda_{\mathcal{E}}t}$. From Theorem \ref{difeq_m} we obtain the following equations for $\nu_n$:
	\[
	\begin{cases}
		\partial_t \nu_1 & = \mathcal{E} \nu_1 - \lambda_{\mathcal{E}}
		\nu_1,                                                                  \\
		\partial_t \nu_n & = \mathcal{E} \nu_n - n\lambda_{\mathcal{E}} \nu_n +
		\delta_0(x)g_n(\nu_1, \dots, \nu_{n-1}),\quad n \geq 2
	\end{cases}
	\]
	with the initial conditions ${\nu_n(0, \cdot, y) = \delta_y(\cdot)}$, ${n \in \mathbb{N}}$.
	
	We define the operator $\mathcal{E}_n$ by setting ${\mathcal{E}_n := \mathcal{E} - n\lambda_{\mathcal{E}}}I$. Since $\lambda_{\mathcal{E}}$ is the largest eigenvalue of $\mathcal{E}$, the spectrum of $\mathcal{E}_n$ for ${n \geq 2}$ is included into ${(-\infty, -(n-1)\lambda_{\mathcal{E}}]}$, that is, it is on the negative semiaxis, since $\lambda_{\mathcal{E}} > 0$.
	
	Further, arguments similar to those given in \cite{limth} in the proof of a similar theorem remain valid.
	
	The value of $n_*$ from the statement of the theorem is determined by the formula ${n_*:= \frac{2\|\mathcal{E}\|}{\lambda_{\mathcal{E}}}}$. The theorem is proved.
\end{proof}

For the number of particles in the case under consideration the following limit theorem is true, the proof of which is carried out according to the scheme of proof of the limit theorem obtained in~\cite{limth}, so we present only the main parts of the proof.

\begin{theorem} \label{limth}
	Let ${\beta^* > \beta_c}$ and $\lambda_{\mathcal{E}} > 0$. If $\beta^{(r)} = O(r!r^{r-1})$ for all sufficiently large ${r \in \mathbb{N}}$, then the following statements hold in the sense of convergence in distribution
	\[ \lim_{t \rightarrow \infty} \mu(t,y) e^{-\lambda_{\mathcal{E}} t} = \xi \psi(y), \qquad \lim_{t \rightarrow \infty} \mu(t) e^{-\lambda_{\mathcal{E}} t} = \xi,
	\]
	where $\psi(y)$ is some non-negative function and $\xi$ is a nondegenerate random variable.
\end{theorem}

\begin{proof}
	Let us define the functions
	\begin{alignat*}{4}
		m(n,x,y)                        & := \lim_{t \rightarrow \infty} \frac{\mathsf{E}_x \mu^n(t,y)}{m_1^n(t,x,y)} &                                & = \lim_{t \rightarrow \infty}
		\frac{m_n(t,x,y)}{m_1^n(t,x,y)} &                                                                             & = \frac{C_n(x,y)}{C_1^n(x,y)},                                                                                             \\
		m(n,x)                          & := \lim_{t \rightarrow \infty} \frac{\mathsf{E}_x \mu^n(t)}{m_1^n(t,x)}     &                                & = \lim_{t \rightarrow \infty} \frac{m_n(t,x)}{m_1^n(t,x)} &  & = \frac{C_n(x)}{C_1^n(x)}.
	\end{alignat*}
	As shown, for example, in \cite{limth}, the functions $C_n(x,y)$ and $C_n(x)$ for ${\beta^* > \beta_c}$ for all ${n \in \mathbb{N}}$ are related by the relation ${C_n(x,y) = \psi^n(y)C_n(x)}$, where $\psi(y)$ is some function, from which the next equalities follow
	\[
	m(n,x,y) = m(n,x) = \frac{C_n(x)}{C_1^n(x)} = \frac{C_n(x,y)}{C_1^n(x,y)}.
	\]
	
	From Theorem~\ref{th8} we have this theorem statements in terms of convergence of the moments of the random variables $\xi(y) =\psi(y)\xi$ and $\xi$.
	
	The distributions of the limit random variables $\xi(y)$ and $\xi$ to be uniquely determined by their moments if the Carleman condition is satisfied
	\[
	\sum_{n = 1}^{\infty} m(n,x,y)^{-1/(2n)} = \infty, \quad \sum_{n = 1}^{\infty} m(n,x)^{-1/(2n)} = \infty.
	\]
	Assuming ${N=1}$ in the notation from \cite{limth} and defining $n_{*}$ as in Theorem~\ref{th8}, we obtain ${C_n(x) \leq \gamma^{n-1}n!n^n}$, where $\gamma$ is some constant, from here and from the estimate $n! \leq ((n+1)/2)^n$ we get
	\[
	\sum_{n = 1}^{\infty} m(n,x)^{-1/(2n)} = \sum_{n = 1}^{\infty} \left(\frac{C_n(x)}{C_1^n(x)} \right)^{-1/(2n)}= \infty.
	\]
	The proof for $m(n,x,y)$ is similar.
	
	Thus, the Stieltjes moment problem has a unique solution, hence the relations from the formulation of the theorem are valid in terms of convergence in distribution. The theorem is proved. \end{proof}

Note that the obtained limit theorem is true without restrictions on the variance of random walk jumps, see~\cite{limth}.

\section{Moments in a Critical Case}
\label{CriticalBRW}

\begin{theorem} \label{thCritical}
	Let ${\beta^* > \beta_c}$ and $\lambda_{\mathcal{E}} = 0$. Then for ${t \rightarrow \infty}$ and all ${n \in \mathbb{N}}$ the following statements hold
	\[
	m_n(t,x,y) \sim J_n(x,y)t^{n-1}, \qquad m_n(t,x) \sim J_n(x) t^{n-1},
	\]
	where $J_n(x,y)$ and $J_n(x)$ are some constants.
\end{theorem}

\begin{proof}
	
	The proof will be carried out for $m_n(t,x,y)$ using the asymptotic relation for the first moment and the equations for the higher moments. The limit relations for $m_n(t,x)$ follow from the form of the integral equations \ref{th_int} and the asymptotics for $m_n(t,x,y)$.
	
	In the case ${\beta^* > \beta_c}$ the operator $\mathcal{E}$ has a unique isolated eigenvalue $\lambda_{\mathcal{E}} = \lambda_0 - b_0$, which is zero in this case, consider its corresponding eigenfunction $f(x) \in l^2(\mathbb{Z}^d)$.
	
	Consider first the second moment $m_2(t,x,y)$, which satisfies the equation
	\[
	\partial_t m_2(t,x,y) = \mathcal{E}m_2(t, x, y) + \delta_0(x)g_2(m_1(t,x,y)).
	\]
	Multiplying this equation scalarly by $f$, we get
	\[
	{\partial_t \langle f, m_2(t,x,y) \rangle} =  {f(0)g_2(m_1(t,0,y))}.
	\]
	Denote ${h(t,y) := \langle f, m_2(t,x,y) \rangle}$, then the function $h(t,y)$ satisfies the equation
	\[
	\partial_t h(t,y) = f(0)g_2(m_1(t,0,y))
	\]
	with the initial condition
	\[
	h(0,y) = \langle f, m_2(0,x,y) \rangle = \langle f, \delta_0(x-y) \rangle = f(y),
	\]
	whose solution has the form
	\[
	h(t,y) = f(y) + \int_{0}^{t} f(0)g_2(m_1(\tau,0,y))\,d\tau.
	\]
	Since for $m_1(t,0,y)$ we have $m_1(t,0,y) \sim C(0,y)$ as ${t \rightarrow \infty}$, then for $h(t,y)$ as ${t \rightarrow \infty}$ the following limit relation holds
	\[
	h(t,y) \sim tf(0)g_2(C(0,y)).
	\]
	Denote by $E_{f}$ the eigensubspace of the operator $\mathcal{E}$ corresponding to the eigenvalue $\lambda_{\mathcal{E}}$, i.e. $E_{f}:= \{tf:~ t\in\mathbb{R}\}$. Via $E_{f}^{\perp}$ we will further denote the orthogonal complement to the subspace $E_{f}$. Then ${l^2(\mathbb{Z}^d) = E_{f}\oplus E_{f}^{\perp}}$, that is, for any ${v \in l^2(\mathbb{Z}^d)}$ there are unique ${\alpha \in \mathbb{C}}$ and ${v_1 \in E_{f}^{\perp}}$ such that ${v = \alpha f + v_1}$. Since $f$ is an eigenfunction of the self-adjoint operator $\mathcal{E}$, then $E_{f}^{\perp}$ is an eigensubspace of the operator $\mathcal{E}$, that is, ${\mathcal{E}E_{f}^{\perp} \subseteq E_{f}^{\perp}}$.
	
	Since $\lambda_{\mathcal{E}} = 0$ is a simple eigenvalue corresponding to the eigenfunction $f$, it is not a point of the spectrum of the operator $\mathcal{E}$ restricted to $E_{f}^{\perp}$, so the spectrum of this operator lies on the negative semiaxis and is separated from zero. Let's use the property, which was noted, for example, in~\cite{YarBRW}: if the spectrum of a self-adjoint continuous operator $\mathcal{H}$ on a Hilbert space is included into $(-\infty, -s]$, $s>0$, and also ${f(t) \rightarrow f_*}$ as $t \rightarrow \infty$, then the solution of the equation
	\[
	{\frac{d\nu}{dt} = \mathcal{H} \nu + f(t)}
	\]
	satisfies ${\nu(t) \rightarrow -\mathcal{H}^{-1}f_*}$ condition.
	
	Since $m_2(t,x,y)$ satisfies the equation
	\[
	\partial_t m_2(t,x,y) = \mathcal{E}m_2(t, x, y) + \delta_0(x)g_2(m_1(t,x,y))
	\]
	and for ${t \rightarrow \infty}$ we have the relation
	\[
	{\delta_0(x)g_2(m_1(t,x,y)) \sim \delta_0(x)g_2(C(x,y)),}
	\]
	we obtain the limit relation that holds on $E_{f}^{\perp}$:
	\[
	m_2(t,x,y)) \sim -\mathcal{E}^{-1}(\delta_0(x)g_2(C(x,y))) =: v_1^*(x,y).
	\]
	We have ${m_2(t,x,y) = \alpha f + v_1}$, where ${\alpha = \frac{\langle f, m_2(t,x,y) \rangle}{\langle f, f \rangle} = \frac{h(t,y)}{\langle f, f \rangle}}$ and $v_1 \sim v_1^*$. For ${t \rightarrow \infty}$ we get the relation
	\[
	m_2(t,x,y) \sim \frac{tf(x)f(0)g_2(C(0,y))}{\langle f, f \rangle}.
	\]
	Denoting ${J_2(x,y) := \frac{f(x)f(0)g_2(C(0,y))}{\langle f, f \rangle}}$, we get that $m_2(t,x,y) \sim J_2(x,y)t.$
	
	Further, we continue similarly, using the asymptotics for the moments obtained at the previous step. On the subspace $E_{f}$, carrying out similar reasoning, for $m_n(t,x,y)$ we obtain the asymptotics
	\[
	m_n(t,x,y) \sim J_n^{(1)} t^{n-1}, \quad t\to \infty,
	\]
	where $J_n^{(1)}$ is some constant. On the subspace $E_{f}^{\perp}$ we use the following property: if the spectrum of a self-adjoint continuous operator $\mathcal{H}$ on a Hilbert space is included into ${(-\infty, -s]}$, ${s>0}$, and ${f(t) = P_n(t)}$, where $P_n(t)$ is a  polynomial of degree $n$, then the solution of the equation
	\[
	{\frac{d\nu}{dt} = \mathcal{H} \nu + f(t)}
	\]
	satisfies ${\nu(t) = Q_n(t) + u(t)}$ condition, where $Q_n(t)$ is a  polynomial of degree $n$ and $u(t)$ is a function that decreases exponentially in $t$. We get that on the subspace $E_{f}^{\perp}$ the asymptotics ${m_n(t,x,y) \sim J_n^{(2)} t^{n-2}}$ is true, where $J_n^{(2)}$ is some constant. So, for $m_n(t,x,y)$ we have
	\[
	m_n(t,x,y) \sim J_n(x,y)t^{n-1}
	\]
	as ${t \rightarrow \infty}$. The theorem is proved. \end{proof}

\section{Moments in a Subcritical Case}
\label{SubcriticalBRW}

To study the asymptotic behavior of the particle number moments for $\lambda_{\mathcal{E}} < 0$ we need an auxiliary lemma.

\begin{lemma}\label{lemm1}
	If the spectrum of a self-adjoint continuous operator $\mathcal{H}$ on a Hilbert space is included into ${(-\infty, -\sigma]}$, ${\sigma>0}$, and $f(t)$ is a function such that $\|f(t)\| < Ce^{-\alpha t}$, where ${C, \alpha > 0}$ are some constants, then the solution of the equation
	\[
	{\frac{d\nu}{dt} = \mathcal{H} \nu + f(t)}
	\]
	satisfies ${\|\nu\| \leq \widetilde{C}_{1} e^{-\min(\alpha, \sigma)t}}$ for ${\alpha \neq \sigma}$ and ${\|\nu\| \leq \widetilde{C}_{2} t e^{-\sigma t}}$ otherwise, where ${\widetilde{C}_{1}, \widetilde{C}_{2}}$ are some constants.
\end{lemma}

\begin{proof}
	The solution of the considered equation with the given initial condition ${\nu(0) = \nu_0}$ can be represented explicitly
	\begin{equation}\label{Enu}
		\nu(t) = e^{\mathcal{H}t}\nu_0 + \int_{0}^{t} e^{\mathcal{H} (t-s)} f(s)\,ds.
	\end{equation}
	Let us estimate the norm of each of the terms. To estimate the norm of the first term we recall some properties of the spectrum of a self-adjoint continuous operator on a Hilbert space, denoting the operator's spectrum as $\spec(\cdot)$.
	
	\begin{itemize}
		\item[1.] \cite[Theorem 7.2.6]{spec}: for any self-adjoint operator $\mathcal{H}$ on a Hilbert space the following equality holds
		\[
		\|\mathcal{H}\| = \sup\{|\lambda|: \lambda \text{~is the point of the spectrum } \mathcal{H}\} .
		\]
		
		\item[2.] \cite[Corollary 7.8.10]{spec}: let $\mathcal{H}$ be a self-adjoint operator and $f$ be a continuous complex function on $\spec(\mathcal{H})$. Then
		\[
		\spec(f(\mathcal{H})) = f(\spec(\mathcal{H})).
		\]
		In particular, $\spec(e^{\mathcal{H}t}) = e^{\spec(\mathcal{H})t}.$
	\end{itemize}
	
	Using these properties, we obtain that the first term in \eqref{Enu} satisfies the estimate
\[{\|e^{\mathcal{H}t}\nu_0\| \leq \|e^{\mathcal{H}t}\|\|\nu_0\| = e^{-\sigma t} \|\nu_0\|.}\]
And for the second term for ${\alpha \neq \sigma}$ we have:
	\begin{align*}
		\left\|\int_{0}^{t} e^{\mathcal{H} (t-s)} f(s)\,ds\right\| & \leq \int_{0}^{t} \left\|e^{\mathcal{H} (t-s)}\right\|\|f(s)\|\,ds \leq \int_{0}^{t} e^{-\sigma (t-s)} Ce^{-\alpha s}\, ds =
		\\
		& = Ce^{-\sigma t} \int_{0}^{t}  e^{(\sigma-\alpha) s}\, ds = \frac{Ce^{-\sigma t}}{\sigma-\alpha} \left.e^{(\sigma-\alpha) s}
		\right|_0^t =                                                                                                                                                                                         \\
		& = \frac{Ce^{-\sigma t}}{-(\sigma-\alpha)}\left(1 - e^{(\sigma - \alpha)t} \right) = \frac{C}{-(\sigma-\alpha)}\left(e^{-\sigma t} - e^{-
			\alpha t} \right) \leq                                                                                                                                                                                \\
		& \leq \widehat{C}e^{-\min(\alpha, \sigma)t}.
	\end{align*}
	In the case $\alpha > \sigma$ we set $\widetilde{C}_{1} = \|\nu_0\| + \widehat{C}$ and in the case $\alpha < \sigma$: $\widetilde{C}_{1} = \widehat{C}$.
	It remains to note that for $\alpha = \sigma$ the following equality holds
	\[
	Ce^{-\sigma t} \int_{0}^{t}  e^{(\sigma-\alpha) s}\, ds = Cte^{-\sigma t},
	\]
	so we can put $\widetilde{C}_{2} = \|\nu_0\| + C$, which completes the proof of Lemma~\ref{lemm1}.
\end{proof}

\begin{theorem} \label{thSubcritical1}
	Let ${\beta^* > \beta_c}$ and $\lambda_{\mathcal{E}} < 0$. Then for ${t \rightarrow \infty}$ and all ${n \in \mathbb{N}}$ the following statements hold
	\[
	m_n(t,x,y) \sim D_n(x,y)e^{\lambda_{\mathcal{E}}t}, \qquad m_n(t,x) \sim D_n(x) e^{\lambda_{\mathcal{E}}t},
	\]
	where $D_n(x,y)$ and $D_n(x)$ are some constants.
	
\end{theorem}

\begin{proof}
	The proof will be carried out for $m_n(t,x,y)$. The limit relations for $m_n(t,x)$ follow from the form of the integral equations \ref{th_int} and the asymptotics for $m_n(t,x,y)$.
	
	As in the proof of Theorem~\ref{thCritical} we consider the eigenfunction ${f(x) \in l^2(\mathbb{Z}^d)}$ with the eigenvalue $\lambda_{\mathcal{E}}$ of the operator $\mathcal{E}$ and denote by $E_{f}^{\perp}$  the subspace in $l^2(\mathbb{Z}^d)$, which orthogonal to the element $f$ (see the corresponding definition in the proof of Theorem~\ref{thCritical}).
	
	Multiplying the equation for $m_2(t,x,y)$ scalarly by $f$, we get
	\[
	{\partial_t \langle f, m_2(t,x,y) \rangle} =  \lambda_{\mathcal{E}}\langle f, m_2(t,x,y) \rangle + {f(0)g_2(m_1(t,0,y))}.
	\]
	Let ${h(t,y) := \langle f, m_2(t,x,y) \rangle}$, this function satisfies the equation
	\[
	\partial_t h(t,y) = \lambda_{\mathcal{E}}h(t,y) + f(0)g_2(m_1(t,0,y))
	\]
	with the initial condition ${h(0,y) = \langle f, m_2(0,x,y) \rangle = \langle f, \delta_0(x-y) \rangle = f(y)}$, whose solution has the form
	\[
	h(t,y) = e^{\lambda_{\mathcal{E}}t}f(y) + \int_{0}^{t} e^{\lambda_{\mathcal{E}}(t-s)} f(0)g_2(m_1(s,0,y))\,ds.
	\]
	
	Since the relation ${m_1(t,0,y) \sim C(0,y)e^{\lambda_{\mathcal{E}}t}}$ holds for $m_1(t,0,y)$, and this and the explicit form of the function $g_2(m_1)$ implies the relation $g_2(m_1(t,0,y)) \sim \widetilde{K} e^{2\lambda_{\mathcal{E}}t}$, where $\widetilde{K}$ is some constant, then $h(t,y)$ satisfies the limit relation
	\[
	h(t,y) \sim K_1(y) e^{\lambda_{\mathcal{E}}t} + K_2 e^{2\lambda_{\mathcal{E}}t},
	\]
	where $K_1(y), K_2$ are constant.
	
	Consider now the subspace $E_{f}^{\perp}$. The function $m_2(t,x,y)$ satisfies the equation
	\[
	\partial_t m_2(t,x,y) = \mathcal{E}m_2(t, x, y) + \delta_0(x)g_2(m_1(t,x,y))
	\]
	and the spectrum of the operator $\mathcal{E}$ restricted to $E_{f}^{\perp}$ is included into ${(-\infty, -\sigma]}$, ${\sigma>0}$. Using Lemma~\ref{lemm1}, we obtain that on the subspace $E_{f}^{\perp}$ for ${-2\lambda_{\mathcal{E}} \neq \sigma}$ the following estimate holds
	\[
	{\|m_2(t,x,y)\| \leq \widetilde{C}_{1}e^{-\min(-2\lambda_{\mathcal{E}}, \sigma)t}}
	\]
	and ${\|m_2(t,x,y)\| \leq \widetilde{C}_{2} t e^{2\lambda_{\mathcal{E}} t}}$ otherwise, with some constants ${\widetilde{C}_{1}, \widetilde{C}_{2}}$.
	
	As in the proof of Theorem~\ref{thCritical}, taking into account the representation $l^2(\mathbb{Z}^d) = E_{f}\oplus E_{f}^{\perp}$, we obtain for $m_2(t,x,y)$ as ${t \rightarrow \infty}$ the relation
	\[
	m_2(t,x,y) \sim D_2(x,y)e^{\lambda_{\mathcal{E}}t}.
	\]
	It remains to note that for all ${n \geq 2}$ and ${t \rightarrow \infty}$ the relation $g_n(m_1, \dots, m_{n-1}) \sim \widetilde{K}_{n} e^{2\lambda_{\mathcal{E}}t}$ holds, where $\widetilde{K}_{n}$ is some constant. This follows from the explicit form of the function ${g_n(m_1, \dots, m_{n-1})}$. And the above reasoning remains true for $m_n(t,x,y)$ for all ${n \in \mathbb{N}}$.
	
	So, for $m_n(t,x,y)$ for all ${n \in \mathbb{N}}$ and for ${t \rightarrow \infty}$ we have
	\[
	m_n(t,x,y) \sim D_n(x,y)e^{\lambda_{\mathcal{E}}t}.
	\]
	The theorem is proved. \end{proof}

Note that in proving Theorems~\ref{thCritical} and \ref{thSubcritical1} in addition to the asymptotic behavior of the first moments, which for $\beta^* > \beta_c$ does not depend on the variance of jumps of the random walk, we also use differential equations for higher moments, which, as noted above, also do not depend on the conditions imposed on the variance of jumps. Consequently, all the results obtained for the case ${\beta^* > \beta_c}$ do not depend on the variance of jumps of the random walk.

To study the asymptotic behavior of the particle number moments in the case $\beta^* \leq \beta_c$, when there is no isolated eigenvalue $\lambda_{\mathcal{E}}$, we need the following auxiliary lemma.

\begin{lemma} \label{lemm2}
	Let continuous functions ${\varphi(t), \chi(t) \geq 0}$, ${t \geq 0}$, satisfy the following asymptotic relations as ${t \rightarrow \infty}$
	\[
	\varphi(t) \sim \varphi_0 t^{\alpha} (\ln t)^{\beta} e^{-b_0 t}, \qquad \chi(t) \sim \chi_0 t^{2 \alpha} (\ln t)^{2 \beta} e^{-2 b_0 t},
	\]
	where $\alpha, \beta \in \mathbb{R}$, $b_0 \in \mathbb{R}_+$ and let $W(t) := \int_{0}^{t} \varphi(t-s)\chi(s)\,ds$.
	Then for $W(t)$ the following asymptotic relation holds as ${t \rightarrow \infty}$
	\[
	W(t) \sim W_0 t^{\alpha} (\ln t)^{\beta} e^{-b_0 t}.
	\]
\end{lemma}

\begin{proof}
	It follows from the form of the asymptotics for the functions $\varphi(t)$ and $\chi(t)$, that for any ${\varepsilon > 0}$ there exists ${\delta > 0}$ such that the following relations hold for ${t \geq \delta}$
	\begin{gather*}
		(1- \varepsilon)t^{\alpha} (\ln t)^{\beta} e^{-b_0 t} \leq \varphi(t) \leq (1 + \varepsilon)t^{\alpha} (\ln t)^{\beta} e^{-b_0 t}, \\ (1- \varepsilon)t^{2 \alpha} (\ln t)^{2 \beta} e^{-2 b_0 t} \leq \chi(t) \leq (1 + \varepsilon)t^{2 \alpha} (\ln t)^{2 \beta} e^{-2 b_0 t}.
	\end{gather*}
	
	We choose ${t \geq 2 \delta}$ and represent the function $W(t)$ as a sum
	\[
	{W(t) = W_{1, \delta}(t) + W_{2, \delta}(t),}
	\]
	where
	\[
	W_{1, \delta}(t) = \int_{0}^{t - \delta} \varphi(t-s)\chi(s)\,ds, \quad W_{2, \delta}(t) = \int_{t - \delta}^{t} \varphi(t-s)\chi(s)\,ds.
	\]
	
	To estimate $W_{1, \delta}(t)$, note that for ${0 \leq s \leq t-\delta}$ the inequality ${t-s \geq \delta}$ holds. Hence we get that
	\begin{multline*}
		\int_{0}^{t - \delta} (1 - \varepsilon)(t-s)^{\alpha} (\ln
		(t-s))^{\beta} e^{-b_0 (t-s)} \chi(s)\,ds \leq W_{1, \delta}(t)\leq\\
		\leq \int_{0}^{t - \delta} (1 + \varepsilon)(t-s)^{\alpha} (\ln (t-s))^{\beta} e^{-b_0 (t-s)} \chi(s)\,ds.
	\end{multline*}
	
	Notice, that
	\begin{multline*}
		\int_{0}^{t - \delta} (t-s)^{\alpha} (\ln (t-s))^{\beta} e^{-b_0
			(t-s)} \chi(s)\,ds = \\
		= e^{-b_0 t} t^{\alpha} (\ln t)^{\beta} \int_{0}^{t - \delta} (1-s/t)^{\alpha} \left(\frac{\ln t + \ln(1 - s/t)}{\ln t}\right)^{\beta} e^{b_0 s} \chi(s)\,ds,
	\end{multline*}
	in this case the functions ${(1-s/t)^{\alpha}}$ and ${\left(\frac{\ln t + \ln(1 - s/t)}{\ln t}\right)^{\beta}}$ tend monotonically to 1 as ${t \rightarrow \infty}$ and ${\vphantom{\int_a^b} e^{b_0 s} \chi(s) \sim \chi_0 s^{2 \alpha} (\ln s)^{2 \beta} e^{-b_0 s}}$ as ${s \rightarrow \infty}$, i.e. ${e^{b_0 s} \chi(s) \in L[0, +\infty).}$
	
	So, we get
	\[
	\int_{0}^{t - \delta} (t-s)^{\alpha} (\ln (t-s))^{\beta} e^{-b_0 (t-s)} \chi(s)\,ds = e^{-b_0 t} t^{\alpha} (\ln t)^{\beta} \left( \int_{0}^{+\infty} e^{b_0 s} \chi(s)\,ds + o(1) \right).
	\]
	
	Consider now $W_{2, \delta}(t)$. Since ${t \geq 2 \delta}$, we have
	\begin{align*}
		W_{2, \delta}(t) & = \int_{t - \delta}^{t} \varphi(t-s)\chi(s)\,ds
		\leq                                                                                         \\
		& \leq (1 + \varepsilon) (t - \delta)^{2 \alpha} (\ln (t-\delta))^{2 \beta}
		e^{-2 b_0 (t - \delta)} \int_{0}^{\delta} \varphi(s)\,ds =                                   \\
		& = e^{-b_0 t} t^{\alpha} (\ln t)^{\beta} o(1).
	\end{align*}
	
	Finally, denoting ${W_0 := \int_{0}^{+\infty} e^{b_0 s} \chi(s)\,ds}$, we obtain the required asymptotic relation and Lemma~\ref{lemm2} is proved.
\end{proof}

\begin{theorem} \label{th_subcritical}
	Let the variance of jumps of the random walk be finite, then for ${t \rightarrow \infty}$ and all ${n \in \mathbb{N}}$ the following statements hold
	
	\begin{itemize}
		\item[$a)$] for $\beta^* = \beta_c$:
		
		\begin{itemize}
			\item[] $d=3$: $m_n(t,x,y) \sim A_n(x,y) t^{-1/2} e^{-b_0 t}$, $m_n(t,x) \sim A_n(x) t^{1/2} e^{-b_0 t}$,
			\item[] $d=4$: $m_n(t,x,y) \sim B_n(x,y) (\ln t)^{-1} e^{-b_0 t}$, $m_n(t,x) \sim B_n(x) t (\ln t)^{-1} e^{-b_0 t}$,
			\item[] $d \geq 5$: $m_n(t,x,y) \sim C_n(x,y) e^{-b_0 t}$, $m_n(t,x) \sim C_n(x) t e^{-b_0 t}$,
		\end{itemize}

		\item[$b)$] for $\beta^* < \beta_c$:
		\begin{itemize}
			\item[] $d \geq 3$:  $m_n(t,x,y) \sim D_n(x,y) t^{-d/2} e^{-b_0 t}$, $m_n(t,x) \sim D_n(x) e^{-b_0 t}$,
		\end{itemize}
	\end{itemize}
	where $A_n(x,y)$, $A_n(x)$, $B_n(x,y)$, $B_n(x)$, $C_n(x,y)$, $C_n(x)$, $D_n(x,y)$ and $D_n(x)$ are some constants.
	
\end{theorem}

\begin{proof}
	The limit relations for the first moments are obtained in Theorem~\ref{th_fm}. The second moments are expressed in terms of the first moments and their convolutions with the functions $g_2(m_1(t,0,y))$ and $g_2(m_1(t,0))$ using the integral equations~\ref{th_int}. Note that the asymptotic relations for the first moments for all $d$ in the case $\beta^* \leq \beta_c$ have the form $m_1 \sim \widetilde{C}_{1} t^{\alpha} (\ln t)^{\beta} e^{-b_0 t}$ and for the functions $g_2(m_1)$ the following asymptotic relations hold: $g_2(m_1) \sim \widetilde{G}_{2} t^{2 \alpha} (\ln t)^{2 \beta} e^{-2b_0 t}$, where $\widetilde{G}_{2}$ is some constant and $\alpha$ and $\beta$ are the same, as in the asymptotics of the corresponding first moment $m_1$. Using Lemma~\ref{lemm2} for the functions $m_1$ and $g_2$, we get that
	\[
	\int_{0}^{t}m_1(t-s)g_2(m_1(s))\,ds \sim W_0 t^{\alpha} (\ln t)^{\beta} e^{-b_0 t}, \quad t\to\infty.
	\]
	Finally, we obtain that for the second moments the relation $m_2 \sim \widetilde{C}_{2} t^{\alpha} (\ln t)^{\beta} e^{-b_0 t}$ holds, i.e. the second moments behave at infinity in the same way as the corresponding first moments, up to a constant. 	
	
	To complete the proof we note that for all ${n \geq 2}$ the following relation will hold
\[g_n(m_1, \dots, m_{n-1}) \sim \widetilde{G}_{n} t^{2 \alpha} (\ln t)^{2 \beta} e^{-2b_0 t}, \quad t\to\infty,\] where $\widetilde{G}_{n}$ is some constant. This means that for all ${n \in \mathbb{N}}$ and ${t \rightarrow \infty}$ the following limit relations will hold: $m_n \sim \widetilde{C}_{n} t^{\alpha} (\ln t)^{\beta} e^{-b_0 t}$. The theorem is proved.
\end{proof}

When the condition~\eqref{HT:def} is satisfied, which leads to an infinite variance of jumps, the following theorem turns out to be true.

\begin{theorem} \label{th_subcritical_ht}
	Under the condition~\eqref{HT:def} for ${t \rightarrow \infty}$ and all ${n \in \mathbb{N}}$ the following statements hold
	\begin{itemize}
		
		\item[$a)$] for ${\beta^* = \beta_c}$:
		\[
		m_n(t,x,y) \sim B_{n,d/\alpha}(x,y) u^*(t), \quad m_n(t,x) \sim B_{n,d/\alpha}(x)  v^*(t),
		\]
		where $B_{n,d/\alpha}(x,y)$ and $B_{n,d/\alpha}(x) > 0$ and
		
		\begin{itemize}
			\item[]  $u^*(t) = t^{d/\alpha - 2}e^{-b_0 t}$, $v^*(t) = t^{d/\alpha - 1}e^{-b_0 t}$, if $d/\alpha \in (1,2)$;
			\item[]  $u^*(t) = (\ln t)^{-1}e^{-b_0 t}$, $v^*(t) = t (\ln t)^{-1} e^{-b_0 t}$, if $d/\alpha = 2$;
			\item[]  $u^*(t) = e^{-b_0 t}$, $v^*(t) = t e^{-b_0 t}$, if $d/\alpha \in (2, +\infty)$;
		\end{itemize}
		
		\item[$b)$] for ${\beta^* < \beta_c}$:
		\[
		m_n(t,x,y) \sim A_{n}(x,y) u^*(t), \quad m_n(t,x) \sim A_{n}(x) v^*(t),
		\]
		where $A_{n}(x,y)$, $A_{n}(x) > 0$, $u^*(t) = t^{-d/\alpha}e^{-b_0 t}$, $v^*(t) = e^{-b_0 t}$.
		
	\end{itemize}
\end{theorem}

\begin{proof}
	Asymptotic relations for the first moments in the case of the condition~\eqref{HT:def} are obtained in Theorem~\ref{th_fm_ht}.
	Note that for all possible values of the parameter $d/\alpha$ for $\beta^* \leq \beta_c$ these relations have the form
	\[
	m_1 \sim C t^{\alpha} (\ln t)^{\beta} e^{-b_0 t},
	\]
	where $\alpha$ and $\beta$ are some known constants.
	
	Further, carrying out the arguments from the proof of the Theorem~\ref{th_subcritical} without changes, we obtain that all integer moments in the case under consideration behave at infinity in the same way as the corresponding first moments, up to a constant. The theorem is proved.
\end{proof}


\end{document}